\documentclass[preprint,11pt]{elsarticle}
\usepackage{}
\usepackage{amsfonts}
\usepackage{amsmath}
\usepackage{amssymb}
\usepackage{latexsym}
\usepackage{graphicx}
\usepackage{ntheorem}
\usepackage{arydshln}
\usepackage[usenames]{color}

\newtheorem{theo}{Theorem}[section]
\newtheorem{theorem}[theo]{Theorem}

\newtheorem{lemma}[theo]{Lemma}
\newtheorem{definition}[theo]{Definition}

\newtheorem{remark}[theo]{Remark}
\theoremstyle{empty}
\theorembodyfont{\rmfamily}
\newtheorem{refproof}{theorem}

\journal{Theoretical Computer Science}

\begin{document}

\begin{frontmatter}

\title{The normalized Laplacian spectrum of subdivisions of a graph}

\author[label01,label3]{Pinchen Xie}
\author[label02,label3]{Zhongzhi Zhang}
\ead{zhangzz@fudan.edu.cn}
\author[label04]{Francesc Comellas}
\ead{comellas@ma4.upc.edu}
\address[label01]{Department of Physics, Fudan
University, Shanghai 200433, China}
\address[label02]{School of Computer Science, Fudan
University, Shanghai 200433, China}
\address[label3]{Shanghai Key Laboratory of Intelligent Information
Processing, Fudan University, Shanghai 200433, China}
\address[label04]{Department of Applied Mathematics IV,
Universitat Polit\`ecnica de Catalunya, 08034 Barcelona Catalonia, Spain}


\begin{abstract}

Determining and analyzing the spectra of graphs is an important and exciting research topic in theoretical computer science.  The eigenvalues of the normalized Laplacian of a graph provide information on its structural properties and also on some relevant dynamical aspects, in particular those related to random walks.  In this paper, we give the spectra of the normalized Laplacian of iterated subdivisions of simple connected graphs. As an example of application of these results we find the exact values of their multiplicative degree-Kirchhoff index, Kemeny's constant and number of spanning trees.
\end{abstract}

\begin{keyword}
Normalized Laplacian spectrum \sep  Subdivision graph \sep Degree-Kirchhoff index \sep Kemeny's constant    \sep Spanning trees
\end{keyword}

\end{frontmatter}


\section{Introduction}

Spectral analysis of graphs has been the subject of considerable research effort in  theoretical computer science~\cite{DaHoMc04,Sp07,TrVuWa13}, due to its wide applications in this area  and in general~\cite{CvSi11,ArCvSiSk12}. In the last few decades a large body of scientific literature has established that important structural and dynamical properties of networked systems are encoded in the eigenvalues and eigenvectors of some matrices associated to their graph representations. The spectra of the adjacency, Laplacian and normalized Laplacian matrices of a graph provide information on bounds on the diameter, maximum and minimum degrees,  possible partitions, and can be used to count the number of paths of a given length,  number of triangles, total number of links, number of spanning trees and many more invariants. Dynamic aspects of a network, like its synchronizability and random walks properties can  be determined from the eigenvalues of the Laplacian and normalized Laplacian matrices which allow also the calculation of some interesting graph invariants like the Kirchhoff index~\cite{GoRo01,Ch97,BrHa12}.

We notice that in the last years there has been an increasing interest in the study of the normalized Laplacian  as many measures for random walks on a network are linked to the eigenvalues and eigenvectors of normalized Laplacian of the associated graph, including the hitting time, mixing time and Kemeny's constant which  can be used as a measure of efficiency of navigation on the network, see~\cite{Lo93,ZhShCh13,KeSn76, LeLo02}.

However, the normalized and standard Laplacian matrices of a network behave quite differently~\cite{ChDaHaLiPaSt04}, and even if the spectrum of one matrix can be determined this does not mean that the other can also be evaluated if the graphs are not regular.  As an example,  the eigenvalues of the Laplacian  of Vicsek fractals can be found analytically~\cite{BlJuKoFe03}, but until now it has not been possible to obtain the spectra of their normalized Laplacian. Thus, the spectra of the standard and normalized Laplacian matrices must be considered independently.

In this paper, we give the spectra of the normalized Laplacian of iterated subdivisions of simple connected graphs and we use these results to find the values of their multiplicative degree-Kirchhoff index, Kemeny's constant and number of spanning trees.



\section{Preliminaries}\label{pre}

Let $G(V,E)$ be any simple connected graph with vertex set $V$ and edge set $E$.
Let $N_0=|V|$ denote the number of vertices of $G$ and $E_0=|E|$ its number of edges.

\begin{definition} (Shirai~\cite{Sh99})
The subdivision graph of $G$, denoted by $s(G)$, is the graph obtained from $G$ by inserting an additional vertex to every edge of $G$.
\end{definition}
We denote $s^0(G)=G$. The $n$-th subdivision of $G$ is obtained through the iteration
$s^n(G)=s(s^{n-1}(G))$ and $N_n$ and $E_n$  denote the total number of vertices and edges of $s^n(G)$. Figure~\ref{Fig.1} illustrates the iterated subdivisions of four-vertex complete graph $K_4$.

From the definition of the subdivision graph, it is obvious that $E_n=2E_{n-1}$ and $N_n=N_{n-1}+E_{n-1}$. Thus, for $n>0$, we have
\begin{equation}\label{ordersizeSn}
\  E_n=2^{n}E_0, \quad\quad
\ N_n=N_0+(2^n-1)E_0.
\end{equation}
Moreover, for any vertex, once it is created, its degree remains unchanged as $n$ grows.
\begin{figure}
\centering
\includegraphics[width=0.8\textwidth]{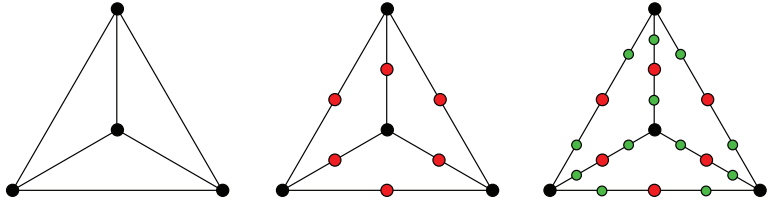}
\caption{$K_4$, $s(K_4)$ and $s^2(K_4)$. Black vertices denote the initial vertices of $K_4$ while red and green vertices are those introduced to obtain $s(K_4)$ and $s^2(K_4)$, respectively.}
\label{Fig.1}
\end{figure}
\begin{definition}
The circuit rank  or cyclomatic number of $G$ is the minimum number $r$ of edges that have to be removed from $G$ to convert the graph into a tree.
\end{definition}
Obviously, the circuit rank of $G$ is $r=E_0-N_0+1$.
\begin{lemma}
The circuit rank of $s^n(G)$  and $G$ are the same for $n\geqslant 0$.
\end{lemma}
\begin{proof}
From the definition of the subdivision graph and~(\ref{ordersizeSn}), the circuit rank of $s^n(G)$ is $r=E_n-N_n + 1= E_0-N_0+1$.
\end{proof}

Given the subdivision graph  $s^n(G)$ we label its nodes from $1$ to $N_n$.	
Let $d_i$ be the degree of vertex $i$, then $D_n = \mbox{diag}(d_1, d_2, \cdots, d_{N_n})$  denotes the diagonal degree matrix of $s^n(G)$  and $A_n$ its adjacency matrix, defined as a matrix with the $(i,j)$-entry  equal to $1$ if vertices $i$ and $j$ are adjacent and $0$ otherwise.
	
We introduce now the probability transition matrix for random walks on  $s^n(G)$ or Markov matrix as $M_n=D_n^{-1}A_n$. $M_n$ can be normalized to obtain a symmetric matrix $P_n$.
\begin{equation}
P_n=D_n^{-\frac{1}{2}} A_n D_n^{-\frac{1}{2}}=D_n ^{\frac{1}{2}}M_n D_n^{-\frac{1}{2}}\,.
\end{equation}
The $(i,j)$th entry of $P_n$ is $(P_n)_{ij}=\frac{A_n(i,j)}{\sqrt{d_id_j}}$.

\begin{definition}
The normalized Laplacian matrix of $s^n(G)$ is
\begin{equation}
{\cal L}_n=I-D_n ^{\frac{1}{2}}M_n D_n^{-\frac{1}{2}}=I-P_n,
\end{equation}
where $I$ is the identity matrix with the same order as $P_n$.
\end{definition}

We denote the spectrum of ${\cal L}_n$ by $\sigma_n=\left\{\lambda^{(n)}_1,\lambda^{(n)}_2,\cdots,\lambda^{(n)}_{N_n}\right\}$.
It is known that $0=\lambda_1^{(n)}<\lambda_2^{(n)}\leqslant\cdots\leqslant\lambda_{N_n-1}^{(n)}\leqslant\lambda_{N_n}^{(n)}\leqslant2$. The spectrum of  the normalized Laplacian matrix of a graph  provides us with relevant structural information about the graph, see~\cite{Ch97,Bu16}.

Below, we will then relate $\sigma_n$ to some significant invariants of $s^n(G)$.

\begin{definition}
Replacing each edge of a simple connected graph $G$ by a unit resistor, we obtain an electrical network $G^*$ corresponding to $G$.	The resistance distance $r_{ij}$ between  vertices $i$ and $j$ of $G$ is equal to the effective resistance between the two equivalent vertices of $G^*$~\cite{KlRa93}.
\end{definition}

\begin{definition}
The multiplicative degree-Kirchhoff index of $G$ is defined as~\cite{ChZh07}:
\begin{equation}
{K\! f}^*(G)=\sum_{i<j}d_id_jr_{ij}.
\end{equation}
\end{definition}
This index is different from the classical Kirchhoff index~\cite{BoBaLiKl94}, ${K\! f}(G)=\sum_{i<j}r_{ij}$, as it takes into account the degree distribution of the graph. 

It has been proved~\cite{ChZh07} that ${K\! f}^*(G)$ can be obtained from the spectrum
$\sigma_0=\{\lambda_1,\lambda_2,\cdots,\lambda_{N_0}\}$ of the normalized Laplacian matrix ${\cal L}_0$ of $G$:
\begin{equation}
	{K\! f}^*(G)=2E_0\sum_{k=2}^{N_0}\frac{1}{\lambda_k},
\end{equation}
where $0=\lambda_1<\lambda_2\leqslant\cdots\leqslant\lambda_{N_0}\leqslant 2$.

Thus,  for $n\geqslant 0$, we have:
\begin{equation}\label{kirchhoff}
	{K\! f}^*(s^n(G))=2E_n\sum_{k=2}^{N_n}\frac{1}{\lambda^{(n)}_k},
\end{equation}
where $0=\lambda_1^{(n)}<\lambda_2^{(n)}\leqslant\cdots\leqslant\lambda_{N_n-1}^{(n)}\leqslant\lambda_{N_n}^{(n)}\leqslant2$ are the eigenvalues of ${\cal L}_n$.

\begin{definition}
	Given a graph $G$, the Kemeny's constant $K(G)$, also known as average hitting time, is the expected number of steps required for the transition from a starting vertex $i$ to a destination vertex, which is chosen randomly according to a stationary distribution of  unbiased random walks on $G$, see~\cite{Hu14} for more details.
\end{definition}


It is known that  $K(G)$ is a constant as it is independent of the selection of starting vertex $i$, see~\cite{LeLo02}. Moreover,  the Kemeny's constant can be computed from  the normalized Laplacian spectrum  in a very simple way as the sum of all reciprocal eigenvalues, except $1/\lambda_1$, see~\cite{Bu16}.
Thus, we can write, for $s^n(G)$ and $\sigma_n$:
\begin{equation}\label{kem}
	K(s^n(G))=\sum_{k=2}^{N_n}\frac{1}{\lambda^{(n)}_k}.
\end{equation}

The last graph invariant considered in this paper is the number of spanning trees of a graph $G$.
A spanning tree is a subgraph of $G$ that includes all the vertices of $G$ and is a tree.
A known result from Chung~\cite{Ch97} allows the calculation of this number from the normalized Laplacian spectrum  and the degrees of all the vertices, thus the number of spanning trees $N_{\rm st}^{(n)}$ of $s^n(G)$ is
\begin{equation}\label{ST01}
N_{\rm st}^{(n)}=\frac{\displaystyle \prod_{i=1}^{N_{n}}
d_i\prod_{i=2}^{N_n}\lambda_i^{(n)}}{\displaystyle \sum_{i=1}^{N_{n}}d_i}\,.
\end{equation}
In the next section we provide an analytical expression for this invariant for any value of $n\geqslant 0$.


\section{Normalized Laplacian spectrum of the subdivision graph $s^n(G)$}

In this section we  find an  analytical expression for the spectrum $\sigma_n$ of ${\cal L}_n(s^n(G))$, the normalized  Laplacian  of the subdivision graph $s^n(G)$. We show that this spectrum can be  obtained iteratively from the spectrum of $G$.
As  ${\cal L}_n=I-P_n$, if $\mu$ is an eigenvalue of $P_n$ then $1-\mu$ is an eigenvalue of ${\cal L}_n$ with the same multiplicity. We denote the multiplicity of $\mu$ as $m_{P_n}(\mu)$. Thus we calculate  first the spectrum $\sigma'_n$ of $P_n$.

\begin{lemma}\label{eig1}
Let $\mu$ be any nonzero eigenvalue of $P_n$ and  let ${R}(x)=2x^2-1$. Then, ${R}(\mu)$ is an eigenvalue of  $P_{n-1}$ with the same multiplicity as $\mu$.
\end{lemma}
\begin{proof}
Divide the vertices of $s^n(G)$ into two groups ${V}_{\rm old}^n$ and ${V}_{\rm new}^n$, where ${V}_{\rm new}^n$ contains all the vertices created when the edges of $s^{n-1}(G)$ are subdivided to generate $s^n(G)$ and  ${V}_{\rm old}^n$ contains the rest. Obviously,  ${V}_{\rm old}^n$ has the same vertices as $s^{n-1}(G)$. Thus for convenience, in the following when any vertex of ${V}_{\rm old}^n$ is considered, it also refers to the corresponding vertex of $s^{n-1}(G)$.

Let $\psi=(\psi_{1},\psi_{2},\ldots,\psi_{N_n})^\top$ be any eigenvector associated to an eigenvalue $\mu$ of $P_n$. Hence
\begin{equation}\label{F1}
P_n\psi=\mu\psi
\end{equation}
Consider a vertex $i\in{V}_{\rm old}^n$ of $s^n(G)$ and denote $\mathcal{N}$ the set of all its neighbors in $s^n(G)$ and $\mathcal{N}'$  the set of all its neighbors in $s^{n-1}(G)$.  From the definition of subdivision graph, there exists a bijection between $\mathcal{N}$ and $\mathcal{N}'$.
If we rewrite Eq.~\eqref{F1} as $\sum_{j=1}^{N_n}(P_n)_{ij}\psi_j=\mu\psi_i,$
we have
\begin{equation}\label{F2}
\mu\psi_i=\sum_{j\in\mathcal{N}}\frac{1}{\sqrt{d_id_j}}\psi_j=\sum_{j\in\mathcal{N}}\frac{1}{\sqrt{2d_i}}\psi_j.
\end{equation}
For any vertex  $ j\in\mathcal{N}$, we have a similar relation
\begin{equation}\label{F3}
\mu\psi_j=\sum_{k=1}^{N_n}(P_n)_{jk}\psi_k=\frac{1}{\sqrt{2d_i}}\psi_i+\frac{1}{\sqrt{2d_{j'}}}\psi_{j'}.
\end{equation}
where vertex $ j'\in\mathcal{N}'$ is the other neighbor of vertex $ j$ in $s^n(G)$ .
Combining Eq.~\eqref{F2} and Eq.~\eqref{F3} yields
\begin{equation}\label{F4}
\begin{split}
\mu\psi_i &=\frac{1}{\mu\sqrt{2d_i}}\times\sum_{ j'\in\mathcal{N}'}\left(\frac{1}{\sqrt{2d_i}}\psi_i+\frac{1}{\sqrt{2d_{j'}}}\psi_{j'}\right) \\
&=\frac{1}{\mu\sqrt{2d_i}}\times\left(\frac{d_i}{\sqrt{2d_i}}\psi_i+\sum_{j'\in\mathcal{N}'}\frac{1}{\sqrt{2d_{j'}}}\psi_{j'}\right).
\end{split}
\end{equation}
Therefore,
\begin{equation}\label{F5}
(2\mu^2-1)\psi_i=\sum_{j'\in\mathcal{N}'}\frac{1}{\sqrt{d_id_{j'}}}\psi_{j'}.
\end{equation}
Eq.~\eqref{F5} directly reflects that ${R}(\mu)=(2\mu^2-1)$ is an eigenvalue of $P_{n-1}$, $\psi_o=(\psi_i)_{i\in {V}_{\rm old}^n}^\top$ is one of its corresponding eigenvectors and $\psi$ can be totally determined by $\psi_o$ by using  Eq.~\eqref{F3}. Hence $m_{P_{n-1}}(R(\mu)) \geqslant  m_{P_n}(\mu)$.

Suppose now that  $m_{P_{n-1}}(R(\mu))> m_{P_n}(\mu)$. Then there exists an extra eigenvector $\psi_e$ associated to ${R}(\mu)$ without an associated eigenvector in $P_n$. But Eq.~\eqref{F3} provides $\psi_e$ with its corresponding eigenvector in $P_n$ as $\mu$ is nonzero, in contradiction with our assumption. Thus  $m_{P_{n-1}}(R(\mu))= m_{P_n}(\mu)$ and the proof is completed.
\end{proof}
\\

\begin{lemma}\label{eig2}
Let $\mu$ be any eigenvalue of $P_{n-1}$ such that $\mu\neq -1$ and let $f_1(x)=\sqrt{\frac{x+1}{2}}$ and $f_2(x)=-\sqrt{\frac{x+1}{2}}$. Then $f_1(\mu)$ and $f_2(\mu)$ are eigenvalues of $P_n$. Besides,
$ m_{P_n}(f_1(\mu))= m_{P_n}(f_2(\mu))=m_{P_{n-1}}(\mu)$.
\end{lemma}
\begin{proof}
This lemma is a direct consequence of Lemma~\ref{eig1}
\end{proof}
\\

\begin{remark}\label{remark1}
Lemmas~\ref{eig1} and~\ref{eig2} indicate that any nonzero eigenvalues of $P_n$ can be obtained from the spectrum  of $P_{n-1}$. Due to the simple expression of $R$, $f_1$ and $f_2$, we easily find that each eigenvalue, except $-1$, of $P_{n-1}$ will generate two unique eigenvalues of $P_n$. Thus $2(N_{n-1}-m_{P_{n-1}}(-1))$ eigenvalues of $\sigma'_n$, and consequently  of  $\sigma_n$, are determined  this way.
\end{remark}

The Perron-Frobenius theorem~\cite{friedland2013perron}  shows that the largest absolute value of the eigenvalues of $P_n$  is always 1. And because of the existence of a unique stationary distribution  for random walks on $s^n(G)$, the multiplicity of the eigenvalue $1$ is always $1$ for any $n\geqslant 0$.  Since $f_2(1)=-1$, we also obtain $m_{P_n}(-1)=1$ for any $n\geqslant 1$. This can be  further explained from the perspective of Markov chains. Since $s^n(G)$ is a bipartite graph~\cite{asratian1998bipartite} containing no odd-length cycles for any $n>0$,  random walks on it are periodic with period $2$, which means it takes an even number of steps to return to the starting vertex. Thus the smallest eigenvalue of the Markov matrix of $s^n(G)$ is $-1$~\cite{zhang2004smallest}. But  random walks on a general graph $G$ can be aperiodic~\cite{diaconis1991geometric} if the graph has an odd-length cycle. Hence the multiplicity of the  eigenvalue $-1$ of $P_0$ depends on the structure of $G$~\cite{zhang2004smallest}.

Lemmas~\ref{eig1} and~\ref{eig2} allow us to obtain the transition between the  eigenvalues of the normalized Laplacian of $s^n(G)$ at each iteration step.
If  $\mu$  is an eigenvalue of  $P_n$ then $\lambda = 1 -\mu$ is an eigenvalue of  ${\cal L}_{n}$. Let $Q(x)=4x-2x^2$, by Lemma~\ref{eig1},  $Q(\lambda)=1-R(1-\lambda)$ is an eigenvalue of  ${\cal L}_{n-1}$ if $\lambda\neq 1$. This allows us to state the following lemma:

\begin{lemma}\label{eigLiter}
Let $\lambda$ be any eigenvalue of ${\cal L}_{n-1}$  such that $\lambda\neq 2$ and let $g_1(x)=1+\sqrt{1-\frac{x}{2}}$ and $g_2(x)=1-\sqrt{1-\frac{x}{2}}$ . Then $g_1(\lambda)$ and $g_2(\lambda)$ are eigenvalues of ${\cal L}_{n}$  and $ m_{{\cal L}_{n}}(g_1(\lambda))= m_{{\cal L}_{n}}(g_2(\lambda))=m_{{\cal L}_{n-1}}(\lambda)$.
\end{lemma}
$\Box$
\\

\begin{definition}
	Let $U=\{u_1,u_2,\cdots,u_k\}$ be any finite multiset of real number where $|u_i|\leqslant 1$ for $i\in [1,k]$. The multisets $\mathcal{R}^{-1 }(U)$ and $\mathcal{Q}^{-1 }(U)$ are defined as
	\begin{equation}
			\mathcal{R}^{-1}(U)=\{f_1(u_1),f_2(u_1),f_1(u_2),f_2(u_2),\cdots,f_1(u_k),f_2(u_k)\};
	\end{equation}
	\begin{equation}
			\mathcal{Q}^{-1}(U)=\{g_1(u_1),g_2(u_1),g_1(u_2),g_2(u_2),\cdots,g_1(u_k),g_2(u_k)\}.
	\end{equation}
\end{definition}

Our main result in this section is the following theorem.

\begin{lemma}\label{theo1}
The  spectrum $\sigma'_n$ of $P_n$,  for $n>0$, is:
 \begin{equation}
 	 \sigma'_n=\mathcal{R}^{-1}\left(\sigma'_{n-1}\setminus \{-1\}\right)\cup\{\underbrace{0, \cdots, 0}_\text{$ m_{P_n}(0)$}\}
 \end{equation}
When $n>1$, $ m_{P_n}(0)=r+1$. When $n=1$, if G contains any odd-length cycle then
$ m_{P_n}(0)=r-1$, otherwise $ m_{P_n}(0)=r+1$, where $r$ is the circuit rank of $G$.
\end{lemma}
\begin{proof}
Combining Lemma~\ref{eig1}, Lemma~\ref{eig2} and considering Remark~\ref{remark1}, the multiplicity of the eigenvalue $0$ of $P_n$ can be determined indirectly:
\begin{equation}
\begin{split}
m_{P_n}(0)&=N_n-2(N_{n-1}-m_{P_{n-1}}(-1)) \\
&=E_0-N_0+2\cdot m_{P_{n-1}}(-1).
\end{split}
\end{equation}

Based on the previous results, it is obvious that $m_{P_{n-1}}(-1)=0$ if and only if $n=1$ and G contains an odd-length cycle, otherwise $m_{P_{n-1}}(-1)=1$, which completes the proof.
\end{proof}

This result allows us to state the main result of this section.

\begin{theo}\label{theo1}
The spectra $\sigma_n$ of ${\cal L}_{n}$ is obtained from $\sigma_{n-1}$ of ${\cal L}_{n-1}$ as:
\begin{equation}
\sigma_n=\mathcal{Q}^{-1}\left(\sigma_{n-1}\setminus \{2\}\right)\cup\{\underbrace{1, \cdots, 1}_\text{$m_{{\cal L}_n}(1)$}\}
\end{equation}
for $n>0$ and where  $m_{{\cal L}_n}(1)=m_{P_n}(0)$ is the multiplicity of the eigenvalue $1$ of ${\cal L}_n$.
\end{theo}
$\Box$

Due to the particularity of the eigenvalue $1$ of ${\cal L}_n$, we call it the exceptional eigenvalue~\cite{bajorin2008vibration} of the family of matrices $\{{\cal L}_n\}$ whose spectra show self-similar characteristics.  For many other family of graphs~\cite{bajorin2008vibration, teplyaev1998spectral} with a similar self-similar property with respect to the spectra of their Markov matrices, the multiplicity of exceptional eigenvalues grows fast as $n$ increases. However, for the normalized Laplacian  of subdivision graphs $\{s^n(G)\}$, the multiplicity of the only exceptional eigenvalue $1$ is always  $r+1$ for $n>1$.


\section{Application of the spectrum of subdivision graph} \label{app}

In this section we use  the spectra of ${\cal L}_n$, the normalized  Laplacian  of the subdivision graph $s^n(G)$, to compute some relevant invariants related to the structure of $s^n(G)$.
Thus, we give closed formulas for the multiplicative degree-Kirchhoff index, Kemeny's constant and the number of spanning trees of $s^n(G)$. These results depend only  on $n$ and some invariants of the original graph $G$.

\subsection{Multiplicative degree-Kirchhoff index}

\begin{theorem}\label{stan}
	The multiplicative degree-Kirchhoff indices of $s^n(G)$ and $s^{n-1}(G)$ are related as follows, for any $n>0$:
	\begin{equation}\label{K}
			{K\! f}^*(s^n(G))=8{K\! f}^*(s^{n-1}(G))+2^n(2r-1)E_0.
	\end{equation}
	Therefore, the general expression for ${K\! f}^*(s^n(G))$ is
	\begin{equation}
			{K\! f}^*(s^n(G))=8^n{K\! f}^*(G)+\frac{8^n-2^n}{3}(2r-1)E_0.
	\end{equation}
\end{theorem}
\begin{proof}
From Eq.~\eqref{kirchhoff} and Corollary~\ref{theo1}, the relation between ${K\! f}^*(s^n(G))$ and ${K\! f}^*(s^{n-1}(G))$ can be expressed as:
\begin{footnotesize}
\begin{equation}\label{K1}
\begin{split}
{K\! f}^*(s^n(G))&=2E_n\left(\frac{1}{\lambda^{(n)}_{N_n}}+\frac{m_{{\cal L}_n}(0)}{1-0}+\sum_{k=2}^{N_{n-1}-1}\left(\frac{1}{g_1(\lambda_k^{(n-1)})}+\frac{1}{g_2(\lambda_k^{(n-1)})}\right)\right)\\
	&=2E_n\left(\frac{1}{2}+(r+1)+4\sum_{k=2}^{N_{n-1}-1}\frac{1}{\lambda_k^{(n-1)}}\right)\\
	&=2E_n\left(\frac{1}{2}+(r+1)+4\left(\frac{K(s^{n-1}(G))}{2E_{n-1}}-\frac{1}{2}\right)\right)\\
	&=8{K\! f}^*(s^{n-1}(G))+2^n(2r-1)E_0,
\end{split}
	\end{equation}
\end{footnotesize}
provided that $m_{{\cal L}_{n-1}}(-1)=1$, where $\lambda_k^{(n)}$ represents the eigenvalue of ${\cal L}_n$.

When $n=1$ and $m_{{\cal L}_n}{0}(-1)=0$,  we obtain:
 \begin{footnotesize}
\begin{equation}\label{K2}
\begin{split}
	{K\! f}^*(s(G))&=2E_1\left(\frac{1}{\lambda^{}_{N_1}}+\frac{m_{{\cal L}_1}(0)}{1-0}+\sum_{k=2}^{N_{0}}\left(\frac{1}{g_1(\lambda_k)}+\frac{1}{g_2(\lambda_k)}\right)\right)\\
	&=2E_1\left(\frac{1}{2}+(r-1)+4\sum_{k=2}^{N_{0}}\frac{1}{\lambda_k}\right)\\
	&=2E_1\left(\frac{1}{2}+(r-1)+4\cdot\frac{K(G)}{2E_0}\right)\\
	&=8{K\! f}^*(G)+2(2r-1)E_0.
\end{split}
\end{equation}
\end{footnotesize}
The similarity of Eq.~\eqref{K1} and Eq.~\eqref{K2} indicates that Eq.~\eqref{K} holds for any simple connected graph $G$ for $n>0$.

Transforming Eq.~\eqref{K} into
\begin{equation}
	{K\! f}^*(s^n(G))+\frac{2^n(2r-1)E_0}{3}=8\left(K(s^{n-1}(G))+\frac{2^{n-1}(2r-1)E_0}{3}\right)
\end{equation}
yields
\begin{equation}
	{K\! f}^*(s^n(G))=8^n{K\! f}^*(G)+\frac{8^n-2^n}{3}(2r-1)E_0.
\end{equation}
\end{proof}

We note here that this expression has been also obtained  recently by Yang and Klein~\cite{yang2015resistance} by using a counting methodology  not related with spectral techniques.
Our result confirms both their calculation and the usefulness of the concise spectral methods described here.

\subsection{Kemeny's constant}
\begin{theorem}\label{kenny}
	The Kemeny's constant for random walks on $s^n(G)$ can be obtained from $K(s^{n-1}(G))$ through
	\begin{equation}
			K(s^n(G))=4K\left(s^{n-1}(G)\right)+r-\frac{1}{2}.
	\end{equation}
	The general expression is
	\begin{equation}
			K(s^n(G))=4^nK(G)+\frac{4^n-1}{3}\left(r-\frac{1}{2}\right).
	\end{equation}
\end{theorem}
\begin{proof}
Theorem~\ref{kenny} is an obvious consequence of Theorem~\ref{stan} considering Eq.~\eqref{kem}.
\end{proof}

\subsection{Spanning trees}

\begin{theorem}\label{sptree}
The number of spanning trees of $s^n(G)$ is,  for any $n>0$:
\begin{equation}
N_{\rm st}^{(n)}=2^rN_{\rm st}^{(n-1)}=2^{rn}N_{\rm st}^{(0)}.
\end{equation}
\end{theorem}

\begin{proof}
From  Eq.~\eqref{ST01} and the properties of the subdivision of a graph:
\begin{equation}\label{st1}
\frac{N_{\rm st}^{(n)}}{N_{\rm st}^{(n-1)}}=2^{E_{n-1}-1}\times\tfrac{\prod\limits_{i=2}^{N_n}\lambda_i^{(n)}}{\prod\limits_{i=2}^{N_{n-1}}\lambda_i^{(n-1)}}.
\end{equation}
Here $\lambda_i^{(n)}$ are the eigenvalues of ${\cal L}_n$

Then, by the same techniques used in the previous subsection, we obtain
\begin{equation}\label{st2}
\begin{split}
	\prod\limits_{i=2}^{N_n}\lambda_i^{(n)}&=2\times\prod\limits_{i=2}^{N_{n-1}-1}\Big(g_1\big(\lambda_i^{(n-1)}\big)g_2\big(\lambda_i^{(n-1)}\big)\Big)\\
	&=2\times\prod\limits_{i=2}^{N_{n-1}-1}\frac{\lambda_i^{(n-1)}}{2}\\
	&=\frac{1}{2^{N_{n-1}-2}}\prod\limits_{i=2}^{N_{n-1}}\lambda_i^{(n-1)}
	\end{split}
\end{equation}
if $m_{{\cal L}_{n-1}}(-1)=1$.

When $n=1$ and $m_{{\cal L}_0}(-1)=1$, the relation becomes
\begin{equation}\label{st3}
\begin{split}
	\prod\limits_{i=2}^{N_1}\lambda_i^{}&=2\times\prod\limits_{i=2}^{N_{0}}\Big(g_1\big(\lambda_i^{}\big)g_2\big(\lambda_i^{}\big)\Big)\\
	&=2\times\prod\limits_{i=2}^{N_{0}}\frac{\lambda_i^{}}{2}\\
	&=\frac{1}{2^{N_{0}-2}}\prod\limits_{i=2}^{N_{0}}\lambda_i
	\end{split}
\end{equation}
which coincides with Eq.~\eqref{st2}.

Therefore, the equality
\begin{equation}\label{st4}
N_{\rm st}^{(n)}=2^{E_{n-1}-N_{n-1}+1}\times N_{\rm st}^{(n-1)}=2^rN_{\rm st}^{(n-1)}=2^{rn}N_{\rm st}^{(0)}
\end{equation}
holds for any $n>0$ as the circuit rank $r$ remains unchanged.
\end{proof}

\section{Conclusion}

In this study we have focused on the analytical calculation of $\sigma_n$, the spectra for the normalized Laplacian of iterated subdivisions of any simple connected graph. This was possible through the analysis of the eigenvectors corresponding to adjacent vertices at different iteration steps.  Our methods could be also applied to find the spectra of other  graph families constructed iteratively.
 	
The  simple relationship between the spectrum of the Markov  and the normalized Laplacian matrices of a subdivision graph  facilitates the calculations to obtain the exact distribution and values of all eigenvalues. One particular interesting result is  that the multiplicity of the exceptional eigenvalues of $\sigma_n$  do not increase exponentially with $n$ as in~\cite{xie2015spectrum,zhang2014full} but  is  a constant determined by the value of the circuit rank of the initial graph.
 	
The calculation of the multiplicative degree-Kirchhoff index, Kemeny's constant and the number of spanning trees in section~\ref{app} is also succinct, if we compare it  to other methods, thanks to the knowledge of the full spectrum of $P_n$ or ${\cal L}_n$. The expressions found can be directly used in the analysis of subdivisions of any simple connected graph while only needing minimal structural information from it,  and this is part of the value of our study.
 	
\section*{Acknowledgements}

This work was supported by the National Natural Science
Foundation of China under grant No. 11275049. F.C. was supported by the
Ministerio de Economia y Competitividad (MINECO), Spain, and the European
Regional Development Fund under project  MTM2014-60127.


\end{document}